\documentclass[12pt]{amsart}

\usepackage{bbm}
\usepackage{mathrsfs}
\usepackage{amsfonts}
\usepackage{amsmath}
\usepackage{amssymb}

\usepackage{graphicx}

\usepackage{latexsym}
\usepackage{cases}
\usepackage{amscd}
\usepackage{amsfonts}

\usepackage[all]{xy}
\usepackage{xy}

\usepackage{amsfonts,amsmath, amssymb}

\newtheorem{theorem}{Theorem}[section]

\newtheorem{corollary}[theorem]{Corollary}
\newtheorem{lemma}[theorem]{Lemma}
\newtheorem{proposition}[theorem]{Proposition}
\newtheorem{remark}[theorem]{Remark}
\newtheorem{definition}[theorem]{Definition}
\newtheorem{example}[theorem]{Example}

\begin{document}

\title{Infinite loop spaces associated to Affine Kac-Moody groups}

\author {Xianzu Lin }

\date{ }
\maketitle
   {\small \it Institute of Mathematics,  Academy of Mathematics and Systems
   Science}\\
    \   {\small \it Chinese Academy of Sciences, Beijing, 100190,
      China}\\
      \              {\small \it Email: linxianzu@126.com}

\begin{abstract}
The main purpose of this paper is to construct infinite loop spaces
from affine Kac-Moody groups, It is well known that to each infinite
class of classical groups over a commutative ring $R$, we can
associate an infinite loop space $G(R)$ by Quillen's plus
construction, in fact it is a functor from the category of
commutative rings to the category of infinite loop spaces. In this
paper we generalize this fact to the cases of affine Kac-Moody
groups. Roughly speaking, there are seven infinite classes of affine
Kac-Moody groups, and to each infinite class we can associate an
analogous functor.

\


\

2000 MR Subject Classification:55P47, 20G44
\end{abstract}

\section{Introduction}
We say that a pointed space $X$ is an $infinite\ loop\ space$ if
there is a sequence of (pointed) spaces $X_{0},X_{1},\cdots$ with
$X_{0}=X$ and weak homotopy equivalences $X_{n}\simeq \Omega
X_{n+1}$.

\begin{example}Let $GL(n)$ be the general linear group over $\mathbb{C}$
and let $BGL$ be the limit of classifying space $\varinjlim BGL(n)$.
By the Bott peoriodicity theorem \cite{a,bot} we have a weak
homotopy equivalence
$$\mathbb{Z}\times BGL\simeq\Omega^{2}(\mathbb{Z}\times BGL);$$
thus $BGL$ is an infinite loop space. Similar results hold for $ BO$
and $BSp$, where $O$ and $Sp$ are the infinite orthogonal and
symplectic group over $\mathbb{C}$ respectively.
\end{example}

In fact we have a very general method of construction. First, we
need some preliminaries.

Let $\Sigma_{n}$ be the symmetric group on the set
$\{1,2,\cdots,n\}$. For any $\sigma\in\Sigma_{m}$ and
$\tau\in\Sigma_{n}$, $\sigma\oplus\tau$ is given by
$\sigma\oplus\tau=$
\begin{subnumcases}
{\sigma\oplus\tau(i)=}
\sigma(i), &$1\leq i\leq m$, \\
m+\tau(i), &$m<i\leq m+n$,
\end{subnumcases}

and $c(m,n)\in\Sigma_{m+n}$ is defined by

\begin{subnumcases}
{c(m,n)(i)=}
n+i, &$1\leq i\leq m$, \\
i-m, &$m<i\leq m+n$.
\end{subnumcases}

The definitions imply $c(m,n)=c(n,m)^{-1}$.

\begin{theorem}\label{th1}
Given a sequence of topological groups
$G(1),G(2),\cdots,G(n),\cdots$ together with homomorphisms
$\phi_{m}:\Sigma_{m}\rightarrow G(m)$, $f_{m}:G(m)\rightarrow
G(m+1)$, $m>0$, satisfying,\\
1) for any $\alpha\in\Sigma_{m}$, we have
$f_{m}\phi_{m}(\alpha)=\phi_{m+1}(\alpha)$;\\
2)set $f_{m,n}:= f_{m+n-1}\cdots f_{m+1}f_{m}$, then
 $\phi_{n}(c(n,m))(f_{n,m}(G(n)))\phi_{n}(c(m,n))$ and $f_{m,n}(G(m))$
are commutative in $G(m+n)$;\\
3)let $G=\lim_{n\rightarrow\infty}G(n)$ and let $\pi'=[\pi,\pi]$ be
the commutator subgroup of $\pi=\pi_{0}(G)$,
we have $\pi'=[\pi',\pi']$.\\
Then $BG^{+}$ (where $+$ means the Quillen's plus construction for
$BG$ and $\pi^{'}\subseteq\pi_{1}(BG)$ ) is an infinite loop space.
\end{theorem}

\begin{proof}
Define a topological category $\xi$ as follows. The objects of $\xi$
are nonnegative integers, $hom_{\xi}(m,n)$ is empty if $m\neq n$ and
$hom_{\xi}(m,m)=G(m)$. One checks that $(\xi,\oplus,0,c)$ has a
structure of permutative category, $M$ is the corresponding
classifying space. The rest of the proof is the same as in
\cite{fp}p.62
\end{proof}

\begin{corollary}
Let $R$ be a commutative ring and set
$$SL(\infty,R)=\lim_{n\rightarrow\infty}SL(n,R),$$ then
$BSL(\infty,R)^{+}$ is an infinite loop space.
\end{corollary}

\begin{proof}
We can easily find natural homomorphisms
$\phi_{n}:\Sigma_{n}\rightarrow SL(2n,R)$, $n>0$ that satisfy the
conditions of the Theorem \ref{th1}.
\end{proof}

Similarly we can show that $BGL(\infty,R)^{+}$, $BO(\infty,R)^{+}$,
$BSO(\infty,R)^{+}$,\\ $BSp(\infty,R)^{+}$ are all infinite loop
spaces. In fact they are functors from the category of commutative
rings to the category of infinite loop spaces.  The main purpose of
this paper is to construct infinite loop spaces from affine
Kac-Moody groups, which are infinite dimensional generalization of
algebraic groups. Roughly speaking, there are seven infinite classes
of affine Kac-Moody groups, and to each infinite class we can
associate an analogous functor.

This paper is structured as follows. \S2 is a short review of
Kac-Moody algebras and Kac-Moody groups, in \S3 we construct the
infinite loop spaces corresponding to affine Kac-Moody groups of
type $A_{2l-1}^{(2)}$, in the final section we consider several
variations and the other cases. Throughout this paper $R$ will be a
fixed commutative ring.

\section{Kac-Moody Algebras and Kac-Moody Groups}
In this section, we give a brief review of Kac-Moody algebras and
Kac-Moody groups, details can be found in \cite{ka,r,ti}.

\begin{definition}
A generalized Cartan matrix is a matrix $A=(a_{i,j})_{i,j=1}^{n}$
satisfying, $a_{i,i}=2$, $a_{i,j}$ are non-positive integers for
$i\neq j$, and $a_{i,j}\neq0$ implies $a_{j,i}\neq0$.
\end{definition}

\begin{definition}
The Kac-Moody algebra $\mathfrak{g}(A)$ associated to a generalized
Cartan matrix $A=(a_{ij})_{i,j=1}^{n}$ is the Lie algebra (over
$\mathbb{C}$) generated by $3n$ elements $e_{i}$, $f_{i}$, $h_{i}$,
$(i=1,\cdots,n)$ with the following defining relations:
   $$[h_{i},h_{j}]=0;\ [h_{i},e_{j}]=a_{ij}e_{j};\ [h_{i},f_{j}]=-a_{ij}f_{j}; [e_{i},f_{j}]=\delta_{i,j}h_{i};$$
   $$(ad\ e_{i})^{1-a_{ij}}e_{j}=0,\
   (ad\ f_{i})^{1-a_{ij}}f_{j}=0,\ if\  i\neq j.$$
\end{definition}

Let $A=(a_{i,j})_{1}^{n}$ be a generalized Cartan matrix. For
$0<i,j\leq n$ set $m_{i,j}=2,3,4\ or \ 6$ if $a_{i,j}a_{j,i}=0,1,2\
or \ 3$ respectively and set $m_{i,j}=0$ otherwise. We associate to
$A$ a discrete group $W(A)$ (the $Weyl$ $group$) on $n$ generators
$s_{1},\cdots,s_{n}$ with relations
$\{(s_{i}s_{j})^{m_{i,j}}=1\}_{0<i,j\leq n}$.

As $ad\ e_{i}$ and $ad\ f_{i}$ are locally nilpotent endomorphisms
of $\mathfrak{g}(A)$, the expressions
$exp(e_{i})=\sum_{n\geq0}\frac{(ad\ e_{i})^{n}}{n!}$ and
$exp(f_{i})=\sum_{n\geq0}\frac{(ad\ f_{i})^{n}}{n!}$ make sense. Set
$s'_{i}=exp(e_{i})exp(-f_{i})exp(e_{i})\in Aut(\mathfrak{g}(A))$ and
let $W'(A)$ be the subgroup of $Aut(\mathfrak{g}(A))$ generated by
the $s'_{i}$. The map $s'_{i}\rightarrow s_{i}$ extends to a group
homomorphism $\phi:W'(A)\rightarrow W(A)$.

Let $V$ be the vector space over $\mathbb{Q}$, with basis
$\{a_{i}\}_{i=1,\cdots,n}$ and let $W(A)$ act on $V$ by
$s_{i}(a_{j})=a_{j}-a_{i,j}a_{i}$. $Real$ $roots$ of
$A=(a_{i,j})_{1}^{n}$ are defined to be elements of $V$ of the form
$w(a_{i})$, with $w\in W(A)$ and $0<i\leq n$. Each real root $a$ is
an integral linear combination of $\{a_{i}\}$, the coefficients of
which of all positive or negative; the real root $a$ is said to be
$positive$ or $negative$ accordingly. Denote by $\triangle$,
$\triangle_{+}$, $\triangle_{-}$ the sets of all real roots,
positive and negative real roots respectively. We say that a set of
real roots $\theta$ is $prenilpotent$ if there exist $w,w'\in W(A)$
such that all elements of $w(\theta)$ are positive and all elements
of $w'(\theta)$ are negative; if, moreover, $a,b\in\theta$ and
$a+b\in\triangle$ imply $a+b\in\theta$, then we said that $\theta$
is $nilpotent$.

For $0<i\leq n$ and $w'\in W'(A)$, the pair of opposite elements
$w'\{e_{i},-e_{i}\}\subset\mathfrak{g}(A)$ depends only on the real
root $a=\phi(w')(a_{i})$ (see \cite{ti} for the proof of this
claim); set $E_{a}=w'\{e_{i},-e_{i}\}$ and denote by $L_{a}$ the
$\mathbb{C}$-subalgebra of $\mathfrak{g}(A)$ generated by $E_{a}$.

For each real root $a$, we denote by $\mathfrak{U}_{a}$ the group
scheme over $\mathbb{Z}$ isomorphic to $Spec\ \mathbb{Z}$ and whose
Lie algebra is the $\mathbb{Z}$-subalgebra of $\mathfrak{g}(A)$
generated by $E_{a}$.

Let $\theta$ be a nilpotent set of real roots, then
$L_{\theta}=\bigoplus_{a\in\theta}L_{a}$ is a nilpotent Lie algebra.
Let $U_{\theta}$ be the unipotent complex algebraic group whose Lie
algebra is $L_{\theta}$. The following proposition was proved in
\cite{ti}.

\begin{proposition}
There exist a uniquely defined group scheme $\mathfrak{U}_{\theta}$
over $\mathbb{Z}$ containing all $\mathfrak{U}_{a}$ for
$a\in\theta$, whose fibre over $\mathbb{C}$ is the group
$U_{\theta}$ and such that for any order on $\theta$, the product
morphism $\prod_{a\in\theta}\rightarrow\mathfrak{U}_{\theta}$ is an
isomorphism of the underlying schemes.
\end{proposition}

Now we present Tits' definition of Kac-Moody group associated to a
generalized Cartan matrix $A=(a_{i,j})_{i,j=1}^{n}$ and a
commutative ring $R$.

Let $\wedge$ be a free abelian group with basis
$h_{1},\cdots,h_{n}$, and $\wedge'$ its dual, then there are $n$
elements $\alpha_{1},\cdots,\alpha_{n}\in\wedge'$ satisfying
$\langle h_{i},\alpha_{j}\rangle=a_{i,j}$. Set
$\mathfrak{T}(R)=Hom(\wedge',R^{*})$. The group $W(A)$ also acts on
$\wedge'$ by
$s_{i}(\lambda)=\lambda-\langle\lambda,h_{i}\rangle\alpha_{i}$. The
automorphism of $\mathfrak{T}(R)$ induced by $s_{i}$ will also
denoted by $s_{i}$.

For a real root $a$, and a nilpotent set of real roots $\theta$, set
$\mathfrak{U}_{a}(R)$, $\mathfrak{U}_{\theta}(R)$ to be the groups
of $R$ points of $\mathfrak{U}_{a}\times Spec\ R$ and
$\mathfrak{U}_{\theta}\times Spec\ R$ respectively. For each pair of
roots $\{a,b\}$, set
$\vartheta(a,b)=(\mathbb{N}a+\mathbb{N}b)\cap\bigtriangleup$.

The $Steinberg$ $group$ $\mathfrak{S}(R)$ over $R$ is defined as the
inductive limit of the groups $\mathfrak{U}_{a}(R)$ and
$\mathfrak{U}_{\vartheta(a,b)}(R)$, where $a\in\bigtriangleup$ and
$\{a,b\}$ runs over all prenilpotent pairs of roots, relative to all
the canonical injections
$\mathfrak{U}_{c}(R)\rightarrow\mathfrak{U}_{\vartheta(a,b)}(R)$ for
$c\in\vartheta(a,b)$. For each $0<i\leq n$,
$s'_{i}:=exp(e_{i})exp(-f_{i})exp(e_{i})$ is an automorphism of
$\mathfrak{g}(A)$ which permutes the $L_{a}$ and the $E_{a}$;
therefore, it induces an automorphism of $\mathfrak{S}(R)$ which we
again denote by $s'_{i}$.

\begin{remark}\label{aaa}
For any $a,b$ in a nilpotent set $\theta$ and any $r,r'\in R$, the
following commutation relation holds inside
$\mathfrak{U}_{\theta}(R)$:
$$[x_{a}(r),x_{b}(r')]=\prod_{c=ma+nb}x_{c}(k(a,b;c)r^{m}r'^{n}),$$
where $c=ma+nb$ runs over $\vartheta(a,b)-\{a,b\}$ and
$x_{a}:R\rightarrow\mathfrak{U}_{a}(R)$,
$x_{b}:R\rightarrow\mathfrak{U}_{b}(R)$ denote the isomorphisms
associated to $a$ and $b$.
\end{remark}

\begin{definition}
The Kac-Moody group $G_{A}(R)$ associated to $A$ over $R$ is defined
to be the quotient of the free product of $\mathfrak{S}(R)$ and
$\mathfrak{T}(R)$ by the following relations.

$$tx_{i}(r)t^{-1}=x_{i}(t(\alpha_{i})r);   \  \
\widetilde{s}_{i}t\widetilde{s}_{i}^{-1}=s'_{i}(t);$$
$$ \widetilde{s}_{i}(r^{-1})=\widetilde{s}_{i}r^{h_{i}}  \  for \
r\in R^{*}\ \ \widetilde{s}_{i}u\widetilde{s}_{i}^{-1}=s'_{i}(u),$$

where $t$ is an element from $\mathfrak{T}(R)$, $r$ is an invertible
element of $R$,  $x_{i}:R\rightarrow\mathfrak{U}_{a_{i}}(R)$ and
$x_{-i}:R\rightarrow\mathfrak{U}_{-a_{i}}(R)$ are the isomorphisms
associated to $e_{i}$ and $f_{i}$ respectively,
$\widetilde{s}_{i}(r)$ is the canonical image of
$x_{i}(r)x_{-i}(r^{-1})x_{i}(r)$ in $\mathfrak{S}(R)$,
$\widetilde{s}_{i}=\widetilde{s}_{i}(1)$, and
$r^{h_{i}}\in\mathfrak{T}(R)$ is defined by
$r^{h_{i}}(\lambda)=r^{\langle\lambda,h_{i}\rangle}$ for
$\lambda\in\wedge'$.
\end{definition}
It is easy to see $G_{A}(R)$ is functorial in $R$, we call $G_{A}$
the $Tits$ $functor$ associated to $A=(a_{ij})_{i,j=1}^{n}$. Set
$r=1$ in $ \widetilde{s}_{i}(r^{-1})=\widetilde{s}_{i}r^{h_{i}}$ we
have $\widetilde{s}_{i}^{2}=(-1)^{h_{i}}$, this formula will be used
in the next section.
\begin{remark}
The above defining relations was given in \cite{r}, and is slightly
different from that of \cite{ti}, in fact the formula
$\widetilde{s}_{i}^{2}=(-1)^{h_{i}}$ cannot be derived from the
defining relations in \cite{ti}.
\end{remark}
\begin{remark}
From the defining relations we see that $G_{A}(R)$ (as a group) is
generated by the image of $\mathfrak{U}_{a_{i}}(R)$ in $G_{A}(R)$.
\end{remark}
In \S3 we need the following lemma.
\begin{lemma}
Let $A$ be a Cartan matrix of type

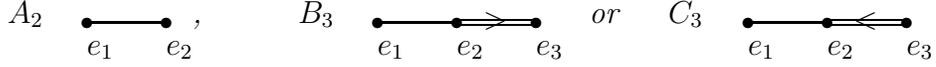
\begin{picture}(100,20)
    \put(30,0){\circle*{4}}
    \put(60,0){\circle*{4}}
    \put(30,0){\line(1,0){30}}
    \put(30,-12){$e_{1}$}
    \put(60,-12){$e_{2}$}
\put(0,0){$A_{2}$} \put(70,0){,}
\put(110,0){$B_{3}$}
 \put(140,0){\circle*{4}}
  \put(170,0){\circle*{4}}
    \put(200,0){\circle*{4}}
\put(140,0){\line(1,0){30}}
    \put(170,1){\line(1,0){30}}
    \put(170,-1){\line(1,0){30}}
\put(140,-12){$e_{1}$}
    \put(170,-12){$e_{2}$}
\put(200,-12){$e_{3}$}
 \put(180,-2.5){$>$}

\put(220,0){or}
\put(250,0){$C_{3}$}
 \put(280,0){\circle*{4}}
  \put(310,0){\circle*{4}}
    \put(340,0){\circle*{4}}
\put(280,0){\line(1,0){30}}
    \put(310,1){\line(1,0){30}}
    \put(310,-1){\line(1,0){30}}
 \put(280,-12){$e_{1}$}
    \put(310,-12){$e_{2}$}
\put(340,-12){$e_{3}$}
 \put(320,-2.5){$<$}
\end{picture}\\

respectively, then the corresponding Kac-Moody group satisfies
$G_{A}(R)=[G_{A}(R),G_{A}(R)]$.
\end{lemma}

\begin{proof}
In the case of $A_{2}$, we have the commutation relation
$[x_{e_{1}}(1),x_{e_{2}}(r)]=x_{e_{1}+e_{2}}(r)$, hence the image of
$\mathfrak{U}_{e_{1}+e_{2}}(R)$ lies in $[G_{A}(R),G_{A}(R)]$. But
the Weyl group acts transitively on the real roots, hence the image
of $\mathfrak{U}_{e_{1}}(R)$ and $\mathfrak{U}_{e_{2}}(R)$ lies in
$[G_{A}(R),G_{A}(R)]$ too. Thus by Remark 2.7, we have
$G_{A}(R)=[G_{A}(R),G_{A}(R)]$.

In the case of $C_{3}$, the above proof shows that the image of
$\mathfrak{U}_{e_{1}}(R)$ and $\mathfrak{U}_{e_{2}}(R)$ lies in
$[G_{A}(R),G_{A}(R)]$. A direct computation shows that in
$\mathfrak{U}_{\vartheta(e_{2},e_{3})}(R)$ we have
$[x_{e_{3}}(r),x_{e_{2}}(1)]=x_{e_{2}+e_{3}}(-r)x_{e_{2}+2e_{3}}(-r)$,
As the Weyl group acts transitively on the set of loot roots, we
have $\mathfrak{U}_{e_{2}+2e_{3}}(R)$ lies in $[G_{A}(R),G_{A}(R)]$
and so is $\mathfrak{U}_{e_{2}+e_{3}}(R)$. But the Weyl group acts
transitively on the set of short roots too, hence
$\mathfrak{U}_{e_{3}}(R)$ also lies in $[G_{A}(R),G_{A}(R)]$. By
Remark 2.7 again, we have $G_{A}(R)=[G_{A}(R),G_{A}(R)]$. The proof
for the case of $B_{3}$ is similar.
\end{proof}

\section{Construction of infinite loop spaces associated to $A_{2l-1}^{(2)}$}
As shown in \cite{ka} there are seven infinite classes of
generalized Cartan matrices of affine type, whose Dynkin diagrams
are listed below.

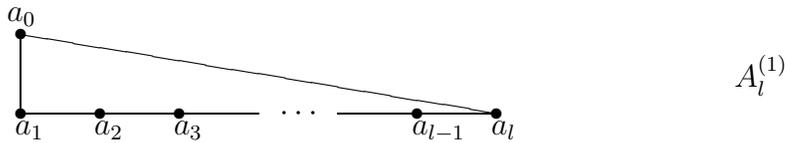
\begin{picture}(200,60)
    \put(0,0){\circle*{4}}
    \put(30,0){\circle*{4}}
    \put(60,0){\circle*{4}}
    \put(0,30){\circle*{4}}
    \put(150,0){\circle*{4}}
    \put(180,0){\circle*{4}}
    \put(0,0){\line(1,0){30}}
     \put(0,0){\line(0,1){30}}
    \put(30,0){\line(1,0){30}}
    \put(60,0){\line(1,0){30}}
    \put(120,0){\line(1,0){30}}
    \put(150,0){\line(1,0){30}}
    \put(0,30){\line(6,-1){180}}
    \put(-2,-8){$a_{1}$}
    \put(28,-8){$a_{2}$}
    \put(58,-8){$a_{3}$}
    \put(-5,35){$a_{0}$}
    \put(148,-8){$a_{l-1}$}
    \put(178,-8){$a_{l}$}
    \put(270,10){$A_{l}^{(1)}$}
     \put(98,-2.8){$\cdots$}
\end{picture}

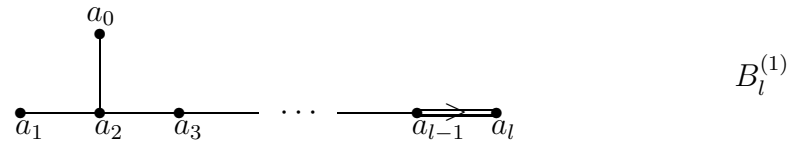
\begin{picture}(200,60)
    \put(0,0){\circle*{4}}
    \put(30,0){\circle*{4}}
    \put(60,0){\circle*{4}}
    \put(30,30){\circle*{4}}
    \put(150,0){\circle*{4}}
    \put(180,0){\circle*{4}}
    \put(0,0){\line(1,0){30}}
    \put(30,0){\line(1,0){30}}
    \put(30,0){\line(0,1){30}}
    \put(60,0){\line(1,0){30}}
    \put(120,0){\line(1,0){30}}
    \put(150,1){\line(1,0){30}}
    \put(150,-1){\line(1,0){30}}
    \put(-2,-8){$a_{1}$}
    \put(28,-8){$a_{2}$}
    \put(58,-8){$a_{3}$}
    \put(25,35){$a_{0}$}
    \put(148,-8){$a_{l-1}$}
    \put(178,-8){$a_{l}$}
    \put(270,10){$B_{l}^{(1)}$}
    \put(160,-2.5){$>$}
     \put(98,-2.8){$\cdots$}
\end{picture}

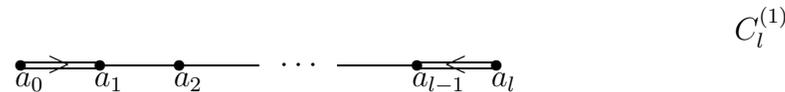
\begin{picture}(200,60)
    \put(0,0){\circle*{4}}
    \put(30,0){\circle*{4}}
    \put(60,0){\circle*{4}}
    \put(150,0){\circle*{4}}
    \put(180,0){\circle*{4}}
    \put(0,1){\line(1,0){30}}
    \put(0,-1){\line(1,0){30}}
    \put(30,0){\line(1,0){30}}
    \put(60,0){\line(1,0){30}}
    \put(120,0){\line(1,0){30}}
    \put(150,1){\line(1,0){30}}
    \put(150,-1){\line(1,0){30}}
    \put(-2,-8){$a_{0}$}
    \put(28,-8){$a_{1}$}
    \put(58,-8){$a_{2}$}
    \put(148,-8){$a_{l-1}$}
    \put(178,-8){$a_{l}$}
    \put(270,10){$C_{l}^{(1)}$}
    \put(10,-2.5){$>$}
    \put(160,-2.5){$<$}
     \put(98,-2.8){$\cdots$}
\end{picture}

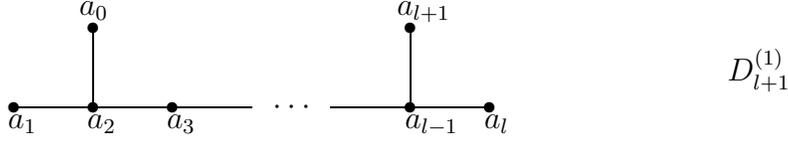
\begin{picture}(200,60)
    \put(0,0){\circle*{4}}
    \put(30,0){\circle*{4}}
    \put(60,0){\circle*{4}}
    \put(30,30){\circle*{4}}
    \put(150,0){\circle*{4}}
    \put(150,30){\circle*{4}}
    \put(180,0){\circle*{4}}
    \put(0,0){\line(1,0){30}}
    \put(30,0){\line(1,0){30}}
    \put(30,0){\line(0,1){30}}
    \put(60,0){\line(1,0){30}}
    \put(120,0){\line(1,0){30}}
    \put(150,0){\line(1,0){30}}
    \put(150,0){\line(0,1){30}}
    \put(-2,-8){$a_{1}$}
    \put(28,-8){$a_{2}$}
    \put(58,-8){$a_{3}$}
    \put(25,35){$a_{0}$}
    \put(148,-8){$a_{l-1}$}
    \put(178,-8){$a_{l}$}
    \put(145,35){$a_{l+1}$}
    \put(270,10){$D_{l+1}^{(1)}$}
     \put(98,-2.8){$\cdots$}
\end{picture}

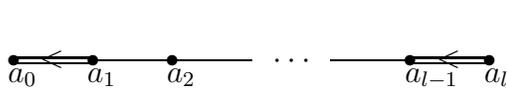
\begin{picture}(200,60)
    \put(0,0){\circle*{4}}
    \put(30,0){\circle*{4}}
    \put(60,0){\circle*{4}}
    \put(150,0){\circle*{4}}
    \put(180,0){\circle*{4}}
    \put(0,1){\line(1,0){30}}
    \put(0,-1){\line(1,0){30}}
    \put(30,0){\line(1,0){30}}
    \put(60,0){\line(1,0){30}}
    \put(120,0){\line(1,0){30}}
    \put(150,1){\line(1,0){30}}
    \put(150,-1){\line(1,0){30}}
    \put(-2,-8){$a_{0}$}
    \put(28,-8){$a_{1}$}
    \put(58,-8){$a_{2}$}
    \put(148,-8){$a_{l-1}$}
    \put(178,-8){$a_{l}$}
    \put(270,10){$A_{2l}^{(1)}$}
    \put(10,-2.5){$<$}
    \put(160,-2.5){$<$}
     \put(98,-2.8){$\cdots$}
\end{picture}

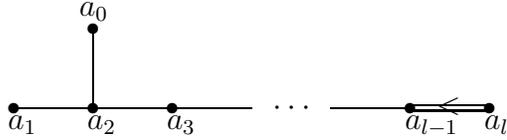
\begin{picture}(200,60)
    \put(0,0){\circle*{4}}
    \put(30,0){\circle*{4}}
    \put(60,0){\circle*{4}}
    \put(30,30){\circle*{4}}
    \put(150,0){\circle*{4}}
    \put(180,0){\circle*{4}}
    \put(0,0){\line(1,0){30}}
    \put(30,0){\line(1,0){30}}
    \put(30,0){\line(0,1){30}}
    \put(60,0){\line(1,0){30}}
    \put(120,0){\line(1,0){30}}
    \put(150,1){\line(1,0){30}}
    \put(150,-1){\line(1,0){30}}
    \put(-2,-8){$a_{1}$}
    \put(28,-8){$a_{2}$}
    \put(58,-8){$a_{3}$}
    \put(25,35){$a_{0}$}
    \put(148,-8){$a_{l-1}$}
    \put(178,-8){$a_{l}$}
    \put(270,10){$A_{2l-1}^{(2)}$}
    \put(160,-2.5){$<$}
     \put(98,-2.8){$\cdots$}
\end{picture}

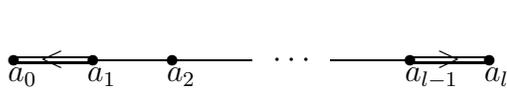
\begin{picture}(200,60)
    \put(0,10){\circle*{4}}
    \put(30,10){\circle*{4}}
    \put(60,10){\circle*{4}}
    \put(150,10){\circle*{4}}
    \put(180,10){\circle*{4}}
    \put(0,11){\line(1,0){30}}
    \put(0,9){\line(1,0){30}}
    \put(30,10){\line(1,0){30}}
    \put(60,10){\line(1,0){30}}
    \put(120,10){\line(1,0){30}}
    \put(150,11){\line(1,0){30}}
    \put(150,9){\line(1,0){30}}
    \put(-2,2){$a_{0}$}
    \put(28,2){$a_{1}$}
    \put(58,2){$a_{2}$}
    \put(148,2){$a_{l-1}$}
    \put(178,2){$a_{l}$}
    \put(270,20){$D_{l+1}^{(2)}$}
    \put(10,7.5){$<$}
    \put(160,7.5){$>$}
     \put(98,7.2){$\cdots$}
\end{picture}

To each infinite class and any commutative ring $R$ we want to
associate a sequence of Kac-Moody groups $G(n)$ that satisfies the
conditions of Theorem \ref{th1}. First consider the case of
$A_{2l-1}^{(2)}$, let $\mathfrak{g}_{l}$ and $G_{l}(R)$ be the
corresponding Kac-Moody algebra and group respectively. In the
following we use the notations of \S2 freely, sometimes the
subscript $l$ will be used to indicate that the notations are
associated to $A_{2l-1}^{(2)}$. For example, $V_{l}$ will be the
vector space over $\mathbb{Q}$, with basis
$\{a_{i}\}_{i=0,\cdots,l}$. The group $W_{l}(A)$ acts on $V_{l}$ and
$\triangle_{l}$ denotes the set of real roots of $A_{2l-1}^{(2)}$

In $\mathfrak{g}_{l+1}$ set $e_{l}'=s'_{l}(e_{l+1})$,
$f_{l}'=s'_{l}(f_{l+1})$, $h_{l}'=s'_{l}(h_{l+1})=h_{l+1}+h_{l}$
respectively and for $i<l$ set $e_{i}'=e_{i}$, $f_{i}'=f_{i}$,
$h_{i}'=h_{i}$ respectively.
\begin{lemma}
In $\mathfrak{g}_{l+1}$ we have, for $i,j\leq l$,
 $$[h_{i}',h_{j}']=0;\ [h_{i}',e_{j}']=a_{ij}e_{j}';\ [h_{i}',f_{j}']=-a_{ij}f_{j}'; [e_{i}',f_{j}']=\delta_{i,j}h_{i}';$$
   $$(ad\ e_{l-1})^{3}e_{l}'=0;\
   (ad\ f_{l-1})^{3}f_{l}'=0.$$
\end{lemma}
\begin{proof}
The first four relations follow from direct computations. Now set
$\mathfrak{g}_{l-1}=\mathbb{C}e_{l-1}\oplus\mathbb{C}f_{l-1}\oplus\mathbb{C}h_{l-1}$
and consider $\mathfrak{g}_{l+1}$ as a $\mathfrak{g}_{l-1}$-module
by restricting of the adjoint representation. Since
$[h_{l-1},e_{l}']=-2e_{l}'$ and $[f_{l-1},e_{l}']=0$ (follows from
the fact that every root is either positive or negative), the
representation theory of $\mathfrak{g}_{0}\cong sl_{2}(\mathbb{C})$
implies $(ad\ e_{l-1})^{3}e_{l}'=0$. The proof for the last relation
is exactly the same.
\end{proof}
By the defining relations of $\mathfrak{g}_{l}$, the map
$e_{i}\rightarrow e'_{i}$, $f_{i}\rightarrow f'_{i}$ extends to an
injective Lie algebra homomorphism
$\varphi_{l}:\mathfrak{g}_{l}\rightarrow\mathfrak{g}_{l+1}$.
\begin{lemma}
Define a linear map $\tau_{l}:V_{l}\rightarrow V_{l+1}$ by
$\tau_{l}(a_{i})=a_{i}$ for $i<l$ and
$\tau_{l}(a_{l})=2a_{l}+a_{l+1}$, then
$\tau_{l}(\triangle_{l}^{\pm})\subset\triangle_{l+1}^{\pm}$ and
$\varphi_{l}(E_{a})=E_{\tau_{l}(a)}$ for any $a\in\triangle_{l}$.
\end{lemma}
\begin{proof}
It is easy to see that the map $s_{i}\rightarrow s_{i}$ for $i<l$
and $s_{l}\rightarrow s_{l}s_{l+1}s_{l}$ extends to a group
homomorphism $w_{l}:W_{l}(A)\rightarrow W_{l+1}(A)$ and for any
$v\in V_{l}$ and $W\in W_{l}(A)$ we have $\tau_{l}\cdot
W(v)=w_{l}(W)\cdot\tau_{l}(v)$. Thus the first assertion follows
readily. Similarly, the map $s'_{i}\rightarrow s'_{i}$ for $i<l$ and
$$s'_{l}\rightarrow
s'_{l}s'_{l+1}(s'_{l})^{-1}=exp(e'_{i})\cdot exp(-f'_{i})\cdot
exp(e'_{i})$$ extends to a group homomorphism
$w'_{l}:W'_{l}(A)\rightarrow W'_{l+1}(A)$, where $W'_{l}(A)\subseteq
Aut(\mathfrak{g}(A)_{l})$ and $W'_{l+1}(A)\subseteq
Aut(\mathfrak{g}(A)_{l+1})$. One checks that $w_{l}$ and $w'_{l}$
are compatible with the homomorphisms $\phi_{l}:W'_{l}(A)\rightarrow
W_{l}(A)$ and $\phi_{l+1}:W'_{l+1}(A)\rightarrow W_{l+1}(A)$. We
also have for any $\omega'\in W'_{l}(A)$,
$\varphi_{l}\cdot\omega'=w'_{l}(\omega')\cdot\varphi_{l}$. Now we
are ready to prove the second assertion. First, it is true for
$a=a_{i}$, $i\leq l$ by the definition of $\varphi_{l}$. Let
$a=\phi_{l}(\omega')(a_{i})$ be an element of $\triangle_{l}$, with
$\omega'\in W'_{l}(A)$, then $\varphi_{l}(E_{a})=\varphi_{l}\omega'(
E_{a_{i}})=w'_{l}(\omega')\varphi_{l}(E_{a_{i}})=w'_{l}(\omega')(E_{\tau_{l}(a_{i})})
=E_{\phi_{l+1}w'_{l}(\omega')(\tau_{l}(a_{i}))}=
E_{w_{l}\phi_{l}(\omega')(\tau_{l}(a_{i}))}
=E_{\tau_{l}(\phi_{l}(\omega')(a_{i}))}=E_{\tau_{l}(a)}$. This
finishes the proof.
\end{proof}
For any $a\in\triangle_{l}$, let $\mathfrak{U}_{a}$ be the
corresponding group scheme defined in \S2, then we can define a
homomorphism
$\psi_{a}:\mathfrak{U}_{a}\rightarrow\mathfrak{U}_{\tau_{l}(a)}$
that is compatible with the map $E_{a}\rightarrow E_{\tau_{l}(a)}$.
\begin{lemma}
Let $\theta\subset\triangle_{l}$ be a nilpotent set of real roots,
then $\tau_{l}(\theta)\subset\triangle_{l+1}$ is also nilpotent; let
$\mathfrak{U}_{\theta}$ and $\mathfrak{U}_{\tau_{l}(\theta)}$ be the
group schemes in Proposition 2.3, then the homomorphism
$\psi_{a}:\mathfrak{U}_{a}\rightarrow\mathfrak{U}_{\tau_{l}(a)}$ for
$a\in\theta$ extends uniquely to a homomorphism
$\psi_{\theta}:\mathfrak{U}_{\theta}\rightarrow\mathfrak{U}_{\tau_{l}(\theta)}$.
\end{lemma}
\begin{proof}By lemma 3.2 the homomorphism
$\varphi_{l}:\mathfrak{g}_{l}\rightarrow\mathfrak{g}_{l+1}$ induces
an isomorphism $L_{\theta}\rightarrow L_{\tau_{l}(\theta)}$. Thus
for $a,b\in\theta$, the commutation relation of $\mathfrak{U}_{a}$
and $\mathfrak{U}_{b}$ in $\mathfrak{U}_{\theta}$ is exactly the
same as that of $\mathfrak{U}_{\tau_{l}(a)}$ and
$\mathfrak{U}_{\tau_{l}(b)}$ in $\mathfrak{U}_{\tau_{l}(\theta)}$.
Now the lemma follows readily.
\end{proof}
By Lemma 3.2 and Lemma 3.3 the group homomorphisms
$\psi_{a}(R):\mathfrak{U}_{a}(R)\rightarrow\mathfrak{U}_{\tau_{l}(a)}(R)$,
$a\in\triangle_{l}$, extend to a group homomorphism
$\psi(R):\mathfrak{S}_{l}(R)\rightarrow\mathfrak{S}_{l+1}(R)$.

Let $\wedge_{l}$ be a free abelian groups with basis
$h_{0},\cdots,h_{l}$ and $\wedge'_{l}$ its dual. Define linear map
$\omega_{l}:\wedge_{l}\rightarrow\wedge_{l+1}$ by
$\omega_{l}(h_{i})=h_{i}$ for $i<l$ and
$\omega_{l}(h_{l})=h_{l}+2h_{l+1}$. Denote by $\omega'_{l}$ the dual
map of $\omega_{l}$, then $\omega'_{l}$ induces a group homomorphism
$\omega_{l}(R):\mathfrak{T}_{l}(R)\rightarrow\mathfrak{T}_{l+1}(R)$.

From the defining relations of Kac-Moody groups and the
constructions of $\psi(R)$ and $\omega_{l}(R)$ we see that the
homomorphism of free products
$\psi\ast\omega_{l}(R):\mathfrak{S}_{l}(R)\ast\mathfrak{T}_{l}(R)\rightarrow
\mathfrak{S}_{l+1}(R)\ast\mathfrak{T}_{l+1}(R)$ reduces to a
homomorphism $g_{l}:G_{l}(R)\rightarrow G_{l+1}(R)$. Set
$G(n):=G_{2n}(R)$ and $f_{n}:=g_{2n+1}\cdot g_{2n}$. In order to
apply Theorem 1.2, we have to define group homomorphism
$\varsigma_{n}:\Sigma_{n}\rightarrow G(n)$ for each $n>0$.

First we need some preliminaries. Let $\overline{W}_{l}$ be the
signed permutation group, i.e., the group of linear transformations
of $\mathbb{R}^{l}$ leaving invariant the set $\{\pm e_{i}\}$ of
standard basis vectors and their negatives. It has $l-1$ generators
$\overline{r}_{1},\cdots,\overline{r}_{l-1}$ and the following
defining relations:
 $$\overline{r}_{j}\overline{r}_{i}^{2}\overline{r}_{j}^{-1}=\overline{r}_{i}^{2}\overline{r}_{j}^{-2a_{i,j}}$$
$$\overline{r}_{i}\overline{r}_{j}\overline{r}_{i}\cdots=\overline{r}_{j}\overline{r}_{i}\overline{r}_{j}\cdots
(m_{i,j}\ factors\ on\ each\ side),$$ where $\overline{r}_{i}$ is
defined by sending $\{e_{i},e_{i+1}\}$ to $\{-e_{i+1},e_{i}\}$ and
leaves the other basis vectors invariant.
\begin{lemma}
The $\widetilde{s}_{i}$, $0<i<l$ in $G_{l}(R)$ satisfy the following
two relations,
$$\widetilde{s}_{j}\widetilde{s}_{i}^{2}\widetilde{s}_{j}^{-1}=\widetilde{s}_{i}^{2}\widetilde{s}_{j}^{-2a_{i,j}},$$
$$\widetilde{s}_{i}\widetilde{s}_{j}\widetilde{s}_{i}\cdots=\widetilde{s}_{j}\widetilde{s}_{i}\widetilde{s}_{j}\cdots
(m_{i,j}\ factors\ on\ each\ side).$$ Let $\widetilde{W}_{l}$ be the
subgroup of $G_{l}(R)$ generated by $\{\widetilde{s}_{i}\}_{0<i<l}$,
then the map $\overline{r}_{i}\rightarrow\widetilde{s}_{i}$ extends
to a group homomorphism
$h_{l}:\overline{W}_{l}\rightarrow\widetilde{W}_{l}$.
\end{lemma}

\begin{proof}
We prove the first assertion and the second assertion will follow
directly. As $\widetilde{s}_{i}^{2}=(-1)^{h_{i}}$ the first relation
is equivalent to
$$\widetilde{s}_{j}(-1)^{h_{i}}\widetilde{s}_{j}^{-1}=(-1)^{h_{i}-2a_{i,j}h_{i}},$$
which is one of the defining relations of $G_{l}(R)$. The second
relation was proved in Remark 3.7 of \cite{ti}.
\end{proof}

Define $w_{i}\in\overline{W}_{2n}$ by sending $\{e_{2i-1},e_{2i}\}$
to $\{e_{2i+1},e_{2i+2}\}$ and leaving the other basis vectors
invariant. set $S_{i}=h_{2n}(w_{i})$, direct computation shows that
$S_{i}=\widetilde{s}_{2i+1}^{3}\widetilde{s}_{2i}\widetilde{s}_{2i-1}\widetilde{s}_{2i+1}\widetilde{s}_{2i}\widetilde{s}_{2i-1}$.

Let $\sigma(i)\in\Sigma_{n}$ be the permutation that swaps the
$i$-th element with the $(i+1)$-th one, then the map
$\sigma(i)\rightarrow S_{i}$ extends to a group homomorphism
$\varsigma_{n}:\Sigma_{n}\rightarrow G(n)=G_{2n}(R)$.
\begin{theorem}
let $G=\lim_{n\rightarrow\infty}G_{n}(R)$, then $\pi=\pi_{0}(G)$
satisfies $\pi=[\pi,\pi]$. Applying Quillen's plus construction
to$BG$ and $\pi^{'}\subseteq\pi_{1}(BG)$, we get an infinite loop
space $BG^{+}$.
\end{theorem}

\begin{proof}
The first assertion follows directly from Lemma 2.8. In order to
apply Theorem 1.2 to $G(n)=G_{2n}(R)$,
$f_{n}=g_{2n+1}g_{2n}:G(n)\rightarrow G(n+1)$ and
$\varsigma_{n}:\Sigma_{n}\rightarrow G(n)=G_{2n}(R)$, we only need
to verify the condition 2) of Theorem 1.2. Set $f_{m,n}=
f_{m+n-1}\cdots f_{m+1}f_{m}$, we want to show that $f_{m,n}(G(m))$
and $c(n,m)(f_{n,m}(G(n)))c(m,n)$ are commutative in $G(m+n)$. Set
$s_{nm}:=\phi_{2m+2n}\varsigma_{m+n}(c(n,m))$ in the following,
recall that $\phi_{2m+2n}$ is the natural homomorphism
$W'(A_{4m+4n-1}^{(2)})\rightarrow W(A_{4m+4n-1}^{(2)})$. \\By remark
2.7, $f_{m,n}(G(m))$ is generated by the subgroups
$\{\mathfrak{U}_{a}(R)\}_{a\in\Theta}$ and
$c(n,m)(f_{n,m}(G(n)))c(m,n)$ is generated by the subgroups
$\{\mathfrak{U}_{a}(R)\}_{a\in\Theta'}$, where $$\Theta=\{\pm
a_{0},\cdots,\pm a_{2m-1}, (s_{2m-1}\cdot s_{2m}\cdots
s_{2m+2n-1})(\pm a_{2m+2n})\}$$$$ =\{\pm a_{0},\cdots,\pm a_{2m-1},
\pm(2a_{2m-1}+\cdots+2a_{2m+2n-1}+a_{2m+2n})\}$$
 and
$$\Theta'=s_{n,m}\{\pm a_{0},\cdots,\pm a_{2n-1},(s_{2n-1}\cdot
s_{2m}\cdots s_{2m+2n-1})(\pm a_{2m+2n})\}.$$ Thus in order to
verify condition (2) it suffices to show that for any
$\alpha\in\Theta$ and $\beta\in\Theta'$, $\mathfrak{U}_{\alpha}(R)$
and $\mathfrak{U}_{\beta}(R)$ are commutative, but this can be
deduced from the fact that the subalgebras $L_{\pm\alpha}$ and
$L_{\pm\beta}$ of $\mathfrak{g}_{2m+2n}$ are commutative. Indeed,
when $L_{\pm\alpha}$ and $L_{\pm\beta}$ are commutative, one checks
that $\{\alpha, \beta\}$ is a prenilpotent pair and
$\vartheta(a,b)=\{\alpha, \beta\}$, hence by Remark \ref{aaa} the
group $\mathfrak{U}_{\vartheta(a,b)}(R)$ is commutative. Thus in
order to finish the proof it suffices to show that for any
$\alpha\in\Theta$ and $\beta\in\Theta'$, $L_{\pm\alpha}$ and
$L_{\pm\beta}$ are commutative.

Direct computation shows that $$(s_{2m-1}\cdot s_{2m}\cdots
s_{2m+2n-1})(\pm a_{2m+2n})=s_{nm}(\pm a_{2m+2n});$$
 $$(s_{2n-1}\cdot
s_{2m}\cdots s_{2m+2n-1})(\pm a_{2m+2n})=s_{mn}(\pm a_{2m+2n});$$
$$s_{mn}(\pm
a_{0})=\pm(a_{0}+a_{1}+2(a_{2}+\cdots+a_{2m})+a_{2m+1});$$
$$s_{nm}\{\pm a_{1},\cdots,\pm a_{2n-1}\}=\{\pm a_{2m+1},\cdots,\pm
a_{2m+2n-1}\};$$
$$s_{nm}\{\pm a_{2n+1},\cdots,\pm
a_{2m+2n-1}\}=\{\pm a_{1},\cdots,\pm a_{2m-1}\}.$$ Thus we only need
to show that $L_{\pm(a_{0}+a_{1}+2(a_{2}+\cdots+a_{2m})+a_{2m+1})}$
is commutative with $L_{\pm a_{0}}$ and $L_{\pm a_{2m+2n}}$ is
commutative with $L_{\pm(2a_{2m-1}+\cdots+2a_{2m+2n-1}+a_{2m+2n})}$.
We proof the first assertion, the proof for the second one is
similar.

Firstly, we have
$[f_{0},e_{a_{0}+a_{1}+2(a_{2}+\cdots+a_{2m})+a_{2m+1}}]\in
L_{a_{1}+2(a_{2}+\cdots+a_{2m})+a_{2m+1}}$, but it is well known
that the highest root in
$\mathbb{Z}a_{1}+\mathbb{Z}a_{2}+\cdots+\mathbb{Z}a_{2m+1}\cap\triangle_{2m+2n}$
is $a_{1}+\cdots+a_{2m+1}$. Hence
$[f_{0},e_{a_{0}+a_{1}+2(a_{2}+\cdots+a_{2m})+a_{2m+1}}]=0$. We also
have $[h_{0},e_{a_{0}+a_{1}+2(a_{2}+\cdots+a_{2m})+a_{2m+1}}]=0$.
 Set
$\mathfrak{g}_{0}=\mathbb{C}e_{0}\oplus\mathbb{C}f_{0}\oplus\mathbb{C}h_{0}$
and consider $\mathfrak{g}_{2m+2n}$ as a $\mathfrak{g}_{0}$-module
by restricting of the adjoint representation. By the representation
theory of $\mathfrak{g}_{0}\cong sl_{2}(\mathbb{C})$, it follows
that $$[e_{0},e_{a_{0}+a_{1}+2(a_{2}+\cdots+a_{2m})+a_{2m+1}}]=0.$$
Similarly we have
$$[e_{0},f_{a_{0}+a_{1}+2(a_{2}+\cdots+a_{2m})+a_{2m+1}}]=0$$ and
$$[f_{0},f_{a_{0}+a_{1}+2(a_{2}+\cdots+a_{2m})+a_{2m+1}}]=0.$$ This
finishes the proof of the theorem. The following Dynkin diagram
would illustrate our proof, where $a_{0}^{'}$ and $a_{2m}^{'}$ are
$2a_{2m-1}+\cdots+2a_{2m+2n-1}+a_{2m+2n}$ and $s_{n,m}( a_{0})$
respectively.

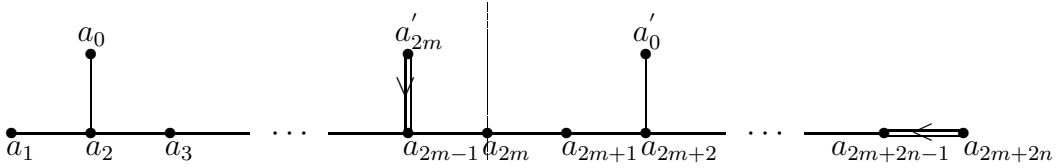
\begin{picture}(200,60)
    \put(0,0){\circle*{4}}
    \put(30,0){\circle*{4}}
     \put(60,0){\circle*{4}}
     \put(30,30){\circle*{4}}
     \put(150,0){\circle*{4}}
     \put(150,30){\circle*{4}}
    \put(180,0){\circle*{4}}
    \put(210,0){\circle*{4}}
    \put(240,0){\circle*{4}}
    \put(240,30){\circle*{4}}
     \put(330,0){\circle*{4}}
    \put(360,0){\circle*{4}}
    \put(0,0){\line(1,0){30}}
    \put(30,0){\line(1,0){30}}
    \put(30,0){\line(0,1){30}}
     \put(151,0){\line(0,1){30}}
     \put(149,0){\line(0,1){30}}
    \put(240,0){\line(0,1){30}}
    \put(60,0){\line(1,0){30}}
    \put(120,0){\line(1,0){30}}
    \put(150,0){\line(1,0){30}}
    \put(180,0){\line(1,0){30}}
     \put(210,0){\line(1,0){30}}
      \put(240,0){\line(1,0){30}}
       \put(300,0){\line(1,0){30}}
      \put(330,1){\line(1,0){30}}
       \put(330,-1){\line(1,0){30}}
     \put(-2,-8){$a_{1}$}
      \put(28,-8){$a_{2}$}
       \put(58,-8){$a_{3}$}
       \put(25,35){$a_{0}$}
       \put(145,35){$a_{2m}^{'}$}
        \put(145,14){$\vee$}
        \put(235,35){$a_{0}^{'}$}
    \put(148,-8){$a_{2m-1}$}
    \put(178,-8){$a_{2m}$}
    \put(208,-8){$a_{2m+1}$}
    \put(238,-8){$a_{2m+2}$}
     \put(310,-8){$a_{2m+2n-1}$}
      \put(360,-8){$a_{2m+2n}$}

     \put(340,-2.5){$<$}
      \put(98,-2.8){$\cdots$}
      \put(278,-2.8){$\cdots$}
 \linethickness{0.1pt}
       \multiput(180,-10)(0,1){60}{\line(0,1){0.5}}
\end{picture}\\

\end{proof}

\section{The constructions in the other cases}
The constructions in the other cases is similar. For example in the
case of $A_{l}^{(1)}$, let $\mathfrak{g}_{l}$ be the Kac-Moody
algebra associated to $A_{l}^{(1)}$, and in $\mathfrak{g}_{l+1}$ set
$e_{l}'=s'_{l}(e_{l+1})$, $f_{l}'=s'_{l}(f_{l+1})$,
$h_{l}'=s'_{l}(h_{l+1})=h_{l+1}+h_{l}$ respectively and for $i<l$
set $e_{i}'=e_{i}$, $f_{i}'=f_{i}$, $h_{i}'=h_{i}$ respectively. In
the case of $D_{l+1}^{(1)}$, set $e_{l}'=s'_{l}\cdot
s'_{l-1}(e_{l+1})$, $f_{l}'=s'_{l}\cdot s'_{l-1}(f_{l+1})$,
$h_{l}'=s'_{l}\cdot s'_{l-1}(h_{l+1})=h_{l+1}+h_{l}+h_{l-1}$
respectively. For the rest constructions we just repeat the
arguments of the previous section.
\begin{remark}
In \S3 we require that $\wedge_{l}$ is freely generated by
$\{h_{0},\cdots,h_{l}\}$, in fact this restriction is not necessary.
For example, in the case of $A_{l}^{(1)}$ we can set $\wedge_{l}$ to
be freely generated by $\{h_{1},\cdots,h_{l}\}$ and add an
$h_{0}:=-h_{1}-\cdots-h_{l}$. When $R$ is a field $K$, the
corresponding Kac-Moody group $G_{l}(K)$ is isomorphic to
$SL_{l+1}(K[t,t^{-1}])$, then $G(\infty,K)^{+}$ is of course an
infinite loop space. However, we don't know the explicit realization
of $G_{l}(R)$ in the general case.
\end{remark}
We can also treat the (topological) affine Kac-Moody groups over
$\mathbb{C}$ (see \cite{ka1} for the definition), and applying the
method of \S2 we have the following result.
\begin{theorem}
Let $\{A_{l}\}_{l>2}$ be one of the seven (infinite) classes of
affine generalized Cartan matrices and let $\{G_{l}\}_{l>2}$ be the
associated simply-connected Kac-Moody groups over $\mathbb{C}$, then
we can define for each $l>2$ a natural homomorphism
$f_{l}:G_{l}\rightarrow G_{l+1}$ such that
$BG=\lim_{l\rightarrow\infty} BG_{l}$ is an infinite loop space.
\end{theorem}

\begin{remark}
In fact there exists a (infinite) classes of classical Lie groups
$\{G(l)\}_{l>2}$ such that $G_{l}$ is isomorphic to a central
extension of the group of polynomial loops or twisted polynomial
loops on $G(l)$.
\end{remark}

\end{document}